\documentclass[a4paper, reqno]{amsart}
\newif\iffinal

\usepackage[utf8x]{inputenc}
\usepackage[T1]{fontenc}
\usepackage{lmodern}
\usepackage[english]{babel}

\usepackage{amssymb}
\usepackage{colonequals}
\usepackage{mathtools}
\usepackage{hyperref}
\usepackage{colortbl}
\usepackage{color}

\DeclareMathOperator{\Int}{Int}
\newcommand{\N}{\mathbb{N}}
\newcommand{\Q}{\mathbb{Q}}
\newcommand{\p}{\mathbb{P}}
\newcommand{\Z}{\mathbb{Z}}
\newcommand{\val}{\mathsf{v}}
\newcommand{\IntZ}{\Int(\Z)}

\newcommand{\fixdiv}{\mathsf{d}}
\newcommand{\cont}{\mathsf{c}}

\newcommand{\Mod}[1]{\ (\mathrm{mod}\ #1)}

\newtheorem{proposition}{Proposition}[section]

\newtheorem{lemma}[proposition]{Lemma}
\theoremstyle{definition}
\newtheorem{definition}[proposition]{Definition}
\newtheorem{example}[proposition]{Example}
\newtheorem{remark}[proposition]{Remark}
\newtheorem{note}[proposition]{Note}

\author{Sarah Nakato}
\address{\parbox{\linewidth}{Institut für Analysis und Zahlentheorie,
  Graz University of Technology\\
    Kopernikusgasse 24, 8010 Graz, Austria}}
\email{\href{mailto:snakato@tugraz.at}{snakato@tugraz.at}}
\thanks{S.~Nakato is supported by the Austrian Science Fund (FWF):
P~30934}

\title[Non-absolutely irreducible elements]{Non-absolutely irreducible elements
in the ring of Integer-valued polynomials}

\keywords{irreducible elements, absolutely irreducible elements, non-absolutely
irreducible elements, integer-valued polynomials}

\subjclass[2010]{13A05, 13B25, 13F20, 11R09, 11C08}

\begin{document}

\begin{abstract}
Let $R$ be a commutative ring with identity. An element $r \in R$ is said to be
absolutely irreducible in $R$ if for all natural numbers $n>1$, $r^n$ has essentially
only one factorization namely $r^n = r \cdots r$.
If $r \in R$ is irreducible in $R$ but for some $n>1$, $r^n$ has other factorizations distinct from
$r^n = r \cdots r$, then $r$ is called non-absolutely irreducible.

In this paper, we construct non-absolutely irreducible elements in the ring
$\IntZ = \{f\in \Q[x] \mid f(\Z) \subseteq \Z\}$ of integer-valued polynomials. We also
give generalizations of these constructions.
\end{abstract}

\maketitle
\section{Introduction}
The ring $\IntZ = \{f\in \Q[x] \mid f(\Z) \subseteq \Z\}$ of integer-valued polynomials
is known not to be a unique factorization domain. To fully understand the factorization
behaviour of $\IntZ$, several researchers have investigated the irreducible elements of
$\IntZ$, see for example \cite{Antoniou-Nakato-Rissner},
\cite{Cahen-Chabert:Elasticity-for-integral-valued-polynomials},
\cite{Chapman-McClain:irrediv-Irreducible-polynomials} and
\cite{Peruginelli:square-free-denominator}.

In \cite[Chapter~7]{Geroldinger-Halter-Koch:Non-unique-factorizations}, Geroldinger and Halter-Koch
defined a type of irreducible elements called absolutely irreducible. They called an
irreducible element $r$ absolutely irreducible if for all natural numbers $n>1$, each
power $r^n$ of $r$ has essentially only one factorization namely $r^n = r \cdots r$.
Such irreducible elements have also been called strong atoms in \cite{Chapman-Krause:Atomic-decay}
and completely irreducible in \cite{Kaczorowski:completely-irreducible}.

Of much interest are the non-absolutely irreducible elements. We call an irreducible element
$r$ non-absolutely irreducible if there exists a natural number $n>1$ such that $r^n$ has other
factorizations essentially distinct from $r \cdots r$.
In \cite{Chapman-Krause:Atomic-decay}, Chapman and Krause proved that the ring of integers of a number field
always has non-absolutely irreducible elements unless it is a unique factorization domain. Similarly,
$\IntZ$ is a non-unique factorization domain with non-absolutely irreducible elements. For instance,
the polynomial $f = \frac{x(x^2+3)}{2}$ is not absolutely irreducible in $\IntZ$ since
\begin{equation*}
f^2 = f \cdot f = \frac{x^2(x^2+3)}{4} \cdot (x^2+3).
\end{equation*}

In this paper, we construct non-absolutely
irreducible elements in $\IntZ$, a first step to characterizing them. The constructions we
give serve as a cornerstone for studying patterns of factorizations in $\IntZ$.

The researchers who have studied factorizations in $\IntZ$
have mostly been considering square-free factorizations. For instance, in \cite{Frisch:prescribed-sets-of-lengths},
Frisch showed that $\IntZ$ has wild factorization behavior but the factorizations she used to realize her main
result (Theorem 9 in \cite{Frisch:prescribed-sets-of-lengths}) were all square-free. It is not known whether
$\IntZ$ exhibits similar behavior for non-square-free factorizations.
The study of the non-absolutely irreducible elements of $\IntZ$ will be helpful in
answering such questions.

We first give some necessary definitions and facts in Section \ref{sect-preliminaries}.
In Sections \ref{sect:typeI} and \ref{sect:typeII}, we construct non-absolutely irreducible
elements in $\IntZ$. We then give a construction for patterns of factorizations in
Section \ref{sect:pattern} and finally in Section \ref{sect:gen-results}, we give
generalizations of the examples in Sections \ref{sect:typeI} and \ref{sect:typeII}.

\section{Preliminaries}\label{sect-preliminaries}
This section contains necessary definitions and facts on factorizations and irreducible elements
of $\IntZ$.

\subsection{Factorization terms}
We only define the factorization terms we need in this paper and refer to
\cite{Geroldinger-Halter-Koch:Non-unique-factorizations} for a deeper study of
factorization theory. Let $R$ be a commutative ring with identity and $r, s \in R$
be non-zero non-units.
\begin{enumerate}
 \item We say $r$ is \emph{irreducible} in $R$ if it cannot be written as the product of
 two non-units of $R$.
 \item A \emph{factorization} of $r$ in $R$ is an expression
       \begin{equation}\label{eq:fac}
        r = a_{1}\cdots a_{n}
       \end{equation}
where $a_i$ is irreducible in $R$ for $1\le i \le n$.
\item The \emph{length} of the factorization in \eqref{eq:fac} is the number $n$ of irreducible
factors.
\item We say $r$ and $s$ are \emph{associated} in $R$ if there exists a unit $u \in R$
such that $r = us$.
 \item Two factorizations
       \begin{equation}\label{eq:2-fac-same-diff}
        r = a_{1}\cdots a_{n} = b_{1} \cdots b_{m}
       \end{equation}
are called \emph{essentially the same} if $n = m$ and after some possible reordering,
$a_{j}$ is associated to $b_{j}$ for $1 \leq j \leq m$. Otherwise, the factorizations in
\eqref{eq:2-fac-same-diff} are called \emph{essentially different}.
 \item An element $r \in R$ is said to be \emph{absolutely irreducible} if it is
 irreducible in $R$ and for all natural numbers $n >1$, every factorization of $r^n$
 is essentially the same as $r^n = r \cdots r$. Equivalently,
 $r \in R$ is called \emph{absolutely irreducible} if $r^n$ has exactly one factorization
 up to associates.

If $r$ is irreducible but there exists a natural number $n>1$ such that $r^n$ has other
factorizations essentially different from $r^n = r \cdots r$,
then $r$ is called \emph{non-absolutely irreducible}.
\end{enumerate}

\subsection{Irreducible elements of $\boldsymbol{\IntZ}$}
We begin with some preliminary definitions and facts, and later state a characterization
of irreducible elements of $\IntZ$ which we shall use in this paper.

\begin{definition}
The ring $\IntZ = \{f \in \Q[x] \mid f(\Z) \subseteq \Z\}$ is called
the \emph{ring of integer-valued polynomials}.
\end{definition}

We refer to \cite{Cahen-Chabert:Integer-valued-polynomials}
for a deeper study of integer-valued polynomials.

\begin{definition}\label{def:fixdiv}
\begin{enumerate}
\item Let $f = \sum_{i = 0}^na_ix^i \in \Z[x]$. The \emph{content} of $f$ is the ideal
\begin{equation*}
\cont(f) = (gcd[a_0, a_1, \ldots, a_n])
\end{equation*}
of $\Z$ generated by the coefficients of $f$. The polynomial $f$ is said to be
\emph{primitive} if $\cont(f) = (1) =\Z$.
\item Let $f\in \IntZ$. The \emph{fixed divisor} of $f$ is the ideal
\begin{equation*}
 \fixdiv(f) = (gcd[f(a) \mid a \in \Z])
\end{equation*}
of $\Z$ generated by the elements $f(a)$ with $a\in \Z$. Note that it is sufficient to
consider $0 \leq a \leq deg(f)$, that is,
\begin{equation*}
 \fixdiv(f) = (gcd[f(a) \mid 0 \leq a \leq deg(f)]),
 \end{equation*}
 cf. \cite[Lemma 2.7]{Anderson-Cahen:Some-factorization-properties}.
 The polynomial $f$ is said to be
\emph{image primitive} if $\fixdiv(f) = (1) =\Z$.
\end{enumerate}
\end{definition}

\begin{note}
\begin{enumerate}
\item A polynomial $\frac{g}{b}$ with $g\in \Z[x]$ and $b \in \N$, is in $\IntZ$ if
and only if $b$ divides the fixed divisor $\fixdiv(g)$ of $g$.
\item Let $f \in \Z[x]$ be primitive with degree $n$ and $p \in \Z$ be prime. If $p$
divides the fixed divisor of $f$, then $p \leq n$, cf. for instance \cite[Remark 3]{Frisch:prescribed-sets-of-lengths}.
\end{enumerate}
\end{note}

\begin{remark}\label{remark:irred:intZ}
In analogy to the well known fact that $f \in \Z[x]$ is irreducible in $\Z[x]$ if and only if it
is primitive and irreducible in $\Q[x]$, Chapman and McClain \cite{Chapman-McClain:irrediv-Irreducible-polynomials}
showed that $f \in \Z[x]$ is irreducible in $\IntZ$ if and only if it is image
primitive and irreducible in $\Q[x]$.
This follows as a special case from Remark \ref{remark:irreducible}.
\end{remark}

\begin{note}\label{note:irreducible}
Every non-zero polynomial $f \in \Q[x]$ can be written in a
unique way up to the sign of $a$ and the signs and indexing of the $g_i$ as
\begin{equation*}
f(x) = \frac{a}{b}\prod_{i \in I}g_i(x)
\end{equation*}
with $a \in \Z, b \in \mathbb{N}$ with gcd$(a, b) = 1$, $I$ a non-empty finite set and
for $i \in I$, $g_i$ primitive and irreducible in $\Z[x]$.
\end{note}

\begin{remark}\cite{Frisch:prescribed-sets-of-lengths}\label{remark:irreducible}
 A non-constant polynomial $f \in \IntZ$ written as
in Note \ref{note:irreducible} is irreducible in $\IntZ$ if and only if:
\begin{enumerate}
\item $a = \pm1$,
\item $(b) = \fixdiv(\prod_{i \in I}g_i)$ and
\item there does not exist a partition of $I$ into non-empty subsets $I = I_1 \uplus I_2$
and $b_1, b_2 \in \mathbb{N}$ with $b_1b_2 = b$ and $(b_1) = \fixdiv(\prod_{i \in I_1}g_i),
(b_2) = \fixdiv(\prod_{i \in I_2}g_i)$.
\end{enumerate}
\end{remark}

\section{Non-absolutely irreducibles: different factorizations of the same length}\label{sect:typeI}
In this section we construct non-absolutely irreducible elements $r$ such that for all $n>1$, the
factorizations of $r^n$ are all of the same length.

Consider the irreducible polynomial
\begin{equation*}
f = \frac{x(x-4)(x^2+3)}{4} \in \IntZ.
\end{equation*}
It can be checked easily that
\begin{equation*}
\fixdiv(x(x-4)(x^2+3)) = (4) = \fixdiv(x^2(x^2+3)) = \fixdiv((x-4)^2(x^2+3)).
\end{equation*}
Furthermore, the polynomials $\frac{x^2(x^2+3)}{4}$ and $\frac{(x-4)^2(x^2+3)}{4}$
are irreducible in $\IntZ$ by Remark \ref{remark:irreducible}. Thus
\begin{equation*}
f^2 = \frac{x^2(x^2+3)}{4} \cdot \frac{(x-4)^2(x^2+3)}{4}
\end{equation*}
is a factorization of $f^2$ and it is essentially different from $f \cdot f$.
Therefore
\begin{equation*}
f = \frac{x(x-4)(x^2+3)}{4}
\end{equation*}
is not absolutely irreducible in $\IntZ$.

More generally, we have the following construction.

\begin{example}\label{exam:typeI}
Let $p$ be an odd prime and $n > 1$ a natural number. Let
\begin{equation*}
h(x) = x^{p^{n-1}(p-1)} - q
\end{equation*}
where $q$ is a prime congruent to $1$ mod $p^{n+1}$ and $q > p^{n-1}(p-1) + n$.
Then $h$ is irreducible in $\Q[x]$ by Eisenstein's irreducibility criterion.

Furthermore, $\val_p(h(u)) \geq n$ for all integers
$u$ not divisible by $p$ since the group of units of $\Z/p^{n}\Z$ is cyclic of order $p^{n-1}(p-1)$.
Moreover, if $r$ is a generator of the group of units of $\Z/p^{n+1}\Z$, then, since $h(r)$ is not zero modulo $p^{n+1}$,
it follows that $\val_p(h(r))= n$. Therefore the minimum $\val_p(h(u))$ for $u$ an integer
not divisible by $p$ is exactly $n$.

Now let $a_1,\ldots, a_n$ be integers divisible by $p$, not representing all
residue classes of $p^2$ that are divisible by $p$ and such that no
$a_i$ (for $1 \leq i \leq n$) is congruent to 0 modulo any prime
$l \leq p^{n-1}(p-1) + n$, $l \neq p$.

We set
\begin{equation*}
f(x)= \frac{h(x)\prod_{i=1}^n (x-a_i)}{p^n}.
\end{equation*}
By the choice of the integers $a_1,\ldots, a_n$, the minimum
$\val_p\left(\prod_{i=1}^n(w-a_i)\right)$ for $w \in p\Z$ is exactly $n$.
Moreover, for each prime $l \leq p^{n-1}(p-1) + n$, $l \neq p$,
$l$ does not divide the fixed divisor of the numerator $h(x)\prod_{i=1}^n (x-a_i)$ of $f(x)$.
Because of these facts and by Remark \ref{remark:irreducible}, $f(x)$ is in $\IntZ$ and it is
irreducible in $\IntZ$.

Now suppose $a_1,\ldots, a_n$ contains at least two
different elements. Then for $k>1$, $f^k$ has factorizations essentially
different from $f \cdots f$.
All of these factorizations have length $k$.

For example, without loss of generality,
let $a_1$ and $a_2$ be different. Then
\begin{equation*}
    f^k= \frac{h(x)(x-a_1)^2\prod_{i=3}^n(x-a_i)}{p^n}
    \cdot \frac{h(x)(x-a_2)^2\prod_{i=3}^n(x-a_i)}{p^n}
    \cdot\underbrace{f \cdots f}_{k-2 \text{ copies}}
\end{equation*}
is a factorization of $f^k$ essentially different from
$f \cdots f$.
\end{example}

\begin{remark}\label{rem:typeI}
In Example \ref{exam:typeI}, we could use a different polynomial $h(x)$, namely:
\begin{equation*}
h(x) = c(x)d(x)
\end{equation*}
where
\begin{eqnarray*}
  c(x) &=& x^{\frac{p^{n-1}(p-1)}{2}} - q \\
  d(x) &=& x^{\frac{p^{n-1}(p-1)}{2}} - r
\end{eqnarray*}
with $q$ and $r$ primes congruent to $1$ and $-1$ respectively,
mod $p^{n+1}$ and both $q, r$ greater than $p^{n-1}(p-1) + n$.
Similarly, both $c(x)$ and $d(x)$ are irreducible in $\Q[x]$ by
Eisenstein's irreducibility criterion.

Furthermore, if $u$ is a unit mod $p^n$, then $\val_p(c(u)) \geq n$ iff
$u$ is a square mod $p^n$ and $\val_p(d(u)) \geq n$ iff
$u$ is a non-square mod $p^n$. Also, if $r$ is a generator of
the group of units of $\Z/p^{n+1}\Z$, then $\val_p(d(r))= n$ and $\val_p(c(r))= 0$.
Therefore the minimum $\val_p(c(u)d(u))$ for $u$ an integer not divisible by $p$
is exactly $n$.

The construction involving two polynomials $c(x)$ and $d(x)$ can be used to
exhibit factorizations of different lengths of a power of an irreducible
polynomial, cf. Example \ref{exam:typeII}.
\end{remark}

\begin{note}
In Example \ref{exam:typeI}, we have one prime in the denominator but this
can be extended to several primes. For instance, if we allow some $a_i$ to be
congruent to 0 modulo other primes
$l < p^{n-1}(p-1) + n$, then the roots of the numerator of $f$ can contain
a complete set of residues modulo some $l$.
More specifically, we have the following example.
\end{note}

\begin{example}
Let $p, q$ be distinct odd primes, and let $n \geq 2q$ be a
natural number. Let
\begin{equation*}
h(x) = x^{p^{n-1}(p-1)} - r
\end{equation*}
where $r$ is a prime congruent to $1$ mod $p^{n+1}$ and
$r > p^{n-1}(p-1) + n$. Then $h$ is irreducible in $\Q[x]$ and
the minimum $\val_p(h(u))$ for $u$ an integer not divisible by
$p$ is $n$.

Let $a_1,\ldots, a_n$ be integers divisible by $p$, not
representing all residue classes of $p^2$ that are divisible by $p$,
and such that:
\begin{enumerate}
\item $a_1,\ldots, a_q$ is a complete system of residues mod $q$, and
the remaining $a_i$ with $i>q$ are all congruent to $1$ mod $q$.
\item For $1 \leq i \leq n$, $a_i \not \equiv 0 \Mod{l}$ for all primes
$l < p^{n-1}(p-1) + n$, $l \neq p, q$.
\end{enumerate}
Set
\begin{equation*}
    f(x) = \frac{h(x)\prod_{i=1}^n (x-a_i)}{q p^n}.
\end{equation*}

Then $f$ is irreducible in $\IntZ$ by Remark \ref{remark:irreducible}
and $f^2$ has a factorization essentially different from $f \cdot f$, namely;

\begin{equation*}
    f^2 =
\frac{h(x)\prod_{i=1}^q (x-a_i)^2 \prod_{i=2q+1}^n (x-a_{i})}{q^2p^n}
\cdot \frac{h(x)\prod_{i=q+1}^{2q}(x-a_{i})^2 \prod_{i=2q+1}^n (x-a_{i})}{p^n}.
\end{equation*}
\end{example}

Also in the spirit of Example \ref{exam:typeI}, we have the following example
involving two primes.

\begin{example}
Let $q < p$ be odd primes, and let $1 < m \leq n$ be natural numbers. Let
\begin{equation*}
t = lcm(q^{m-1}(q-1), p^{n-1}(p-1)).
\end{equation*}
 We set
\begin{equation*}
h(x) = x^{t} - r
\end{equation*}
where $r$ is a prime congruent to $1$ mod $p^{n+1}q^{m+1}$ and
$r > t + n$. Then $h$ is irreducible in $\Q[x]$ and
$\val_p(h(u)) \geq n$ for all integers $u$ not divisible by
$p$, and $\val_q(h(w)) \geq m$ for all integers $w$ not divisible by
$q$.

Now let $a_1,\ldots, a_n$ be integers divisible by $p$ but
not representing all residue classes of $p^2$ that are divisible
by $p$, and such that:

\begin{enumerate}
\item $a_1,\ldots, a_m$ are divisible by $q$ but not representing all
residue classes of $q^2$ that are divisible by $q$, and
the remaining $a_i$ with $i>m$ are all congruent to $1$ mod $q$.
\item For $1 \leq i \leq n$, $a_i \not \equiv 0 \Mod{l}$ for all primes
$l < t + n$, $l \neq p, q$.
\end{enumerate}

We set
\begin{equation*}
f(x)= \frac{h(x)\prod_{i=1}^n (x-a_i)}{p^nq^m}.
\end{equation*}
Then $f$ is irreducible in $\IntZ$ by Remark \ref{remark:irreducible} and if
$a_1,\ldots, a_m$ or $a_{m+1},\ldots, a_n$ contains at least two different elements,
then for some $k>1$, $f^k$ has a factorization essentially different from
$f \cdots f$.

For instance, without loss of generality let $a_1$ and $a_2$ be different. Then
\begin{equation*}
    f^2= \frac{h(x)(x-a_1)^2\prod_{i=3}^n(x-a_i)}{p^nq^m}
    \cdot\frac{h(x)(x-a_2)^2\prod_{i=3}^n(x-a_i)}{p^nq^m}
\end{equation*}
is a factorization of $f^2$ essentially different from $f \cdot f$. Similarly, if
$a_{m+1}$ and $a_{m+2}$ are different, then

\begin{equation*}
    f^2= \frac{h(x)g(x)(x-a_{m+1})^2}{p^nq^m}
    \cdot\frac{h(x)g(x)(x-a_{m+2})^2}{p^nq^m}
\end{equation*}
where $g(x) = \prod_{i=1}^m(x-a_i)\prod_{i=m+3}^n(x-a_i)$, is a factorization of
$f^2$ essentially different from $f \cdot f$.
\end{example}

\section{Non-absolutely irreducibles: factorizations of different lengths}\label{sect:typeII}
Here we construct non-absolutely irreducible elements $r$ such that for some $n>1$, some
factorizations of $r^n$ have different lengths.

Consider the irreducible polynomial
\begin{equation*}
f = \frac{(x-3)(x^3-17)(x^3-19)}{3} \in \IntZ.
\end{equation*}
Then
\begin{equation*}
f^2 = \frac{(x-3)^2(x^3-17)(x^3-19)}{9} \cdot (x^3-17)(x^3-19)
\end{equation*}
is a factorization of $f^2$ essentially different from $f \cdot f$.
This results from
\begin{equation*}
gcd(a^3 -17 \mid a \in \{2+3\Z\}) = gcd(a^3 -19 \mid a \in \{1+3\Z\}) = 9
\end{equation*}
such that for all $a \not \equiv 0 \Mod{3}$, $(a^3 -17)(a^3 -19)$ is divisible by $9$.

This behaviour motivates the next example and more generally Lemma \ref{lemma:type2_non-abs}.

\begin{example}\label{exam:typeII}
Let $p$ be an odd prime and $n > m$ be natural numbers.
We set
\begin{eqnarray*}
  c(x) &=& x^{\frac{p^{n-1}(p-1)}{2}} - q \\
  d(x) &=& x^{\frac{p^{n-1}(p-1)}{2}} - r
\end{eqnarray*}
where $q$ and $r$ are primes congruent to $1$ and $-1$ respectively,
mod $p^{n+1}$ and both $q, r$ are greater than $p^{n-1}(p-1) + m$.
Then both $c(x)$ and $d(x)$ are irreducible in $\Q[x]$ by
Eisenstein's irreducibility criterion.

Furthermore, if $u$ is a unit mod $p^n$, then $\val_p(c(u)) \geq n$ iff
$u$ is a square mod $p^n$ and $\val_p(d(u)) \geq n$ iff
$u$ is a non-square mod $p^n$.
Note that both $c(x)$ and $d(x)$ are irreducible in
$\IntZ$ by Remark \ref{remark:irred:intZ}
because, being primitive, they are irreducible in $\Z[x]$ and
$\fixdiv(c(x)) = \fixdiv(d(x)) = (1)$. Furthermore, $\fixdiv(c(x)d(x)) = (1)$.

Now let $a_1,\ldots, a_m$ be integers divisible by $p$, not representing all
residue classes of $p^2$ that are divisible by $p$ and such that no $a_i$
(for $1 \leq i \leq m$) is congruent to 0 modulo any prime
$l \leq p^{n-1}(p-1) + m$, $l \neq p$.

Set
\begin{equation*}
f(x)= \frac{c(x)d(x)\prod_{i=1}^m (x-a_i)}{p^m}.
\end{equation*}
Then $f$ is irreducible in $\IntZ$ by Remark \ref{remark:irreducible}.
Now irrespective of all $a_1,\ldots, a_m$ being the same or different,
\begin{equation*}
    f^n= \prod_{i=1}^m\frac{c(x)d(x)(x-a_i)^n}{p^{n}}\cdot c(x)^{n-m}d(x)^{n-m}
\end{equation*}
is a factorization of $f^n$ essentially different from
$\underbrace{f \cdots f}_{n \text{ copies}}$.
\end{example}

Note that $f^k$ can have factorizations essentially
different from $f \cdots f$ also for $k < n$, see for example
the proof of Lemma \ref{lemma:type2_non-abs}.

For our next general example, we begin with the following motivation.

\begin{example}\label{exam:typeIIc}
Consider the irreducible polynomial
\begin{equation*}
f = \frac{(x^4 + x^3 +8)(x-3)}{4} \in \IntZ.
\end{equation*}
It can easily be checked that
\begin{equation*}
gcd(a^4 + a^3 +8 \mid a \in \{0+2\Z\}) = 8
\end{equation*}
and
\begin{equation*}
gcd(a^4 + a^3 +8 \mid a \in \{1+2\Z\}) = 2.
\end{equation*}
Thus $0$ and $1$ are both roots mod 2 of $x^4 + x^3 +8$ and
for all $a \in \{0+2\Z\}$, $a^4 + a^3 +8$ is divisible by 8.

Therefore
\begin{equation*}\label{fact:typeIIc}
f^2 = \frac{(x^4 + x^3 +8)(x-3)^2}{8} \cdot \frac{(x^4 + x^3 +8)}{2}
\end{equation*}
is a factorization of $f^2$ essentially different from $f \cdot f$.
\end{example}

We need the following lemma for our general example.

\begin{lemma}\cite[Lemma 6]{Frisch:prescribed-sets-of-lengths},
\cite[ Lemma 3.3]{Frisch-Nakato-Rissner:sets-of-lengths}\label{lemma:replacements}
Let $I \neq \emptyset$ be a finite set and $f_i\in \Z[x]$ be monic polynomials
for $i\in I$. Then there exist monic polynomials $F_i \in \Z[x]$ for $i \in I$,
such that
\begin{enumerate}
\item $\deg(F_i) = \deg(f_i)$ for all $i\in I$,
\item the polynomials $F_i$ are irreducible in $\Q[x]$ and
   pairwise non-associated in $\Q[x]$ and
\item for all subsets $J\subseteq I$ and all partitions
  $J = J_1 \uplus J_2$,
  \begin{equation*}
    \fixdiv\left(\prod_{j\in J_1} f_j\prod_{j\in J_2}F_j\right)
    = \fixdiv\left(\prod_{j\in J} f_j\right).
 \end{equation*}
 \label{repl-prop-3}
\end{enumerate}
\end{lemma}

\begin{example}\label{exam:typeIIb}
Let $p>3$ be a prime number and let $a_1,\ldots, a_p$ be a complete set of
residues mod $p$ that does not contain a complete set of residues
mod any prime $q<p$. Let

\begin{eqnarray*}
g_1 = (x-a_2)^2(x-a_3)^2\prod_{i=4}^{p}(x-a_i)\\
g_2 = (x-a_1)^2(x-a_3)^2\prod_{i=4}^{p}(x-a_i)\\
g_3 = (x-a_1)^2(x-a_2)^2\prod_{i=4}^{p}(x-a_i).
\end{eqnarray*}

By Lemma \ref{lemma:replacements}, we find polynomials $G_1, G_2, G_3$, of
the same degree as $g_1, g_2, g_3$ respectively, irreducible in $\Q[x]$
and pairwise non-associated in $\Q[x]$ such that for any product $P$ of
polynomials from among the $g_i$ and any product $Q$ that differs from $P$
in that some of the $g_i$ have been replaced by their respective $G_i$,
we have $\fixdiv(P) = \fixdiv(Q)$.

Let $e_p(g) = \val_p(\fixdiv(g))$ denote the exponent of $p$ in the fixed
divisor of $g$. Now note that for each index $i$, $e_p(G_i)=0$ and
for any two different indices $i,j$, $e_p(G_i G_j)=2$, and, finally,
$e_p(G_1 G_2 G_3)=3$.

This shows that
\begin{equation*}
f = \frac{G_1 G_2 G_3}{p^3}
\end{equation*}
is in $\IntZ$ and is irreducible in $\IntZ$, and that $f^2$ factors as
\begin{equation}\label{fact:type2}
f^2 =\frac{G_1 G_2}{p^2} \cdot \frac{G_2 G_3}{p^2} \cdot \frac{G_3 G_1}{p^2},
\end{equation}
which factorization is essentially different from $f \cdot f$. Thus $f$ is not
absolutely irreducible.
\end{example}

Note that in the above example, $p$ divides the fixed divisor of $G_iG_j$ for $i \neq j$
and
\begin{equation*}
\val_p\left(gcd\left(G_i(a)G_j(a)\Biggm\rvert a \equiv a_k \Mod{p}, k\neq i,j \right)\right)
= 4 > \val_p(\fixdiv(f)).
\end{equation*}
This behaviour is similar to the one in Example \ref{exam:typeIIc} and more generally in
Lemma \ref{lemma:type2i_non-abs}.

\begin{remark}
\begin{enumerate}
\item Like Example \ref{exam:typeI}, the constructions in this section can be
extended to several primes in the denominator of $f$.
\item In Example \ref{exam:typeII}, if we employ the usual
\begin{equation*}
h(x) = x^{p^{n-1}(p-1)} - q
\end{equation*}
where $q$ is a prime congruent to $1$ mod $p^{n+1}$ and $q > p^{n-1}(p-1) + n$,
instead of $c(x)d(x)$, $f$ remains non-absolutely irreducible but the
factorizations of $f^k$ all have the same length $k$.
\item In the examples we have given in this section, the factorizations of $f^k$ have
length greater than or equal to $k$ but we can also have
non-absolutely irreducibles $f$ in $\IntZ$ such that for some $k>2$, some factorizations of
$f^k$ have length less than $k$. For instance, consider the polynomial
\begin{equation*}
f = \frac{(x^2+4)(x^4 +7)}{4} \in \IntZ.
\end{equation*}
It is clearly irreducible and
\begin{equation*}
f^3 = \frac{(x^2+4)^3(x^4 +7)^2}{64} \cdot (x^4 +7)
\end{equation*}
is a factorization of $f^3$ essentially different from $f \cdot f \cdot f$ and it is of length 2.
\end{enumerate}
\end{remark}

\section{Patterns of factorizations}\label{sect:pattern}
The researchers who have studied factorizations in $\IntZ$
have mostly been considering square-free factorizations. For instance, in \cite{Frisch:prescribed-sets-of-lengths},
Frisch showed that $\IntZ$ has wild factorization behavior but the factorizations she used to realize her main
result (Theorem 9 in \cite{Frisch:prescribed-sets-of-lengths}) were all square-free. It is not known whether
$\IntZ$ exhibits similar behavior for non-square-free factorizations.
The study of non-absolutely irreducible elements lays a foundation for
studying patterns of factorizations.

As a first step to understanding patterns of factorizations in $\IntZ$, we give a construction using the
examples in Sections \ref{sect:typeI} and \ref{sect:typeII}. We begin with a motivation.

\begin{definition}\label{Defin:partition}
Let $R$ be a commutative ring with identity and $r \in R$ be a nonzero non-unit.
\begin{enumerate}
 \item A sequence of natural numbers $\lambda = (k_1,\ldots, k_s)$ is called a
 \emph{partition of a natural number} $n$ if $ k_1+ \cdots+ k_s = n$ with
 $k_1\ge k_2\ge \cdots\ge k_s>0$. The natural numbers $k_1, \ldots, k_s$ are called \emph{blocks}.
 \item If $\lambda = (k_1,\ldots, k_s)$ is a partition, we say a factorization of $r$ is of
 \emph{type} $\lambda$ if $r = a_1^{k_1} \cdots a_s^{k_s}$ for pairwise non-associated
 irreducible elements $a_1,\dots, a_s \in R$.
\end{enumerate}
\end{definition}

\begin{example}Consider the different partitions of 4:
\begin{equation*}
\{(4), (3,1), (2,2), (2,1,1), (1,1,1,1)\}.
 \end{equation*}
The polynomial
\begin{equation*}
f = \frac{(x^{8}-17)^4(x-4)^2(x-8)^2}{2^4} \in \IntZ
\end{equation*}
gives us factorizations of type $\lambda$ for partitions $\lambda$ of 4
other than (4):
 \begin{eqnarray*}
   f &=&(x^{8}-17)^3 \cdot \frac{(x^{8}-17)(x-4)^2(x-8)^2}{2^4} \\
     &=&(x^{8}-17)^2 \cdot \left(\frac{(x^{8}-17)(x-4)(x-8)}{2^2}\right)^2  \\
     &=&(x^{8}-17)^2 \cdot \frac{(x^{8}-17)(x-4)^2}{2^2} \cdot \frac{(x^{8}-17)(x-8)^2}{2^2}\\
     &=&\frac{(x^{8}-17)(x-4)}{2} \cdot \frac{(x^{8}-17)(x-8)}{2}
     \cdot \frac{(x^{8}-17)(x-4)(x-8)}{2^2} \cdot (x^{8}-17).
\end{eqnarray*}
\end{example}
Note, however, that $f$ has factorizations other than those above.
For example,
\begin{equation*}
f = (x^{8}-17)^2 \cdot \frac{(x^{8}-17)(x-4)^2(x-8)}{2^3} \cdot \frac{(x^{8}-17)(x-8)}{2}
\end{equation*}
is another factorization of $f$ essentially different from the above.

More generally, we have the following construction for patterns of factorizations in $\IntZ$.
We first give a remark.

\begin{remark}
In the following example, we use partition of sets; if a set $S$ is the disjoint union of $m$
non-empty subsets $B_1,\ldots, B_m$, then, we call $B = \{B_1,\ldots, B_m\}$ a partition of $S$.
This should not be confused with the concept of partition of a number as defined in Definition \ref{Defin:partition}.
\end{remark}

\begin{example}\label{example-pattern}
Let $p \in \Z$ be an odd prime and $n, s, t > 1$ natural numbers. Set
\begin{eqnarray*}
  c_i(x) &=& x^{\frac{p^{n-1}(p-1)}{2}} - q_i \\
  d_i(x) &=& x^{\frac{p^{n-1}(p-1)}{2}} - r_i
\end{eqnarray*}
where $q_1,\ldots, q_s$ are primes congruent to $1$ modulo $p^{n+1}$,
$r_1, \ldots, r_t$ are primes congruent to $-1$ modulo $p^{n+1}$ and
for all $1 \leq i \leq s$ and $1 \leq j \leq t$,
$q_i, r_j >p^{n-1}(p-1) + n$.

Now let $a_1,\ldots, a_n$ be integers divisible by $p$, not representing
all residue classes of $p^2$ that are divisible by $p$ and such that no
$a_j$ (for $1 \leq j \leq n$) is congruent to 0 modulo any prime
$l \leq p^{n-1}(p-1) + n$, $l \neq p$. Set
\begin{equation*}
   G(x) = \frac{\prod_{i=1}^s c_i(x)\prod_{i=1}^t d_i(x) \prod_{j=1}^n (x-a_j)}{p^n}.
\end{equation*}

Then every factorization of $G$ in $\IntZ$ corresponds to a triple $(B, \theta, \sigma)$
where:
\begin{enumerate}
\item $B$ is a partition of the set $\{1,\ldots, n\}$ into $m_B$ blocks $B_1,\ldots, B_{m_B}$,
\item $\theta$ is an injective function $\theta:\{1,\ldots, m_B\} \rightarrow \{1,\dots, s\}$ and
\item $\sigma$ is an injective function $\sigma:\{1,\ldots, m_B\} \rightarrow \{1,\dots, t\}$.
\end{enumerate}

Given such a triple, for $1 \leq i \leq m_B$, we construct a polynomial $g_i$
corresponding to the $i$-th block. Suppose the $i$-th block consists of $w_i$ elements.
We set

\begin{equation*}
g_i = \frac{c_{\theta(i)}(x)d_{\sigma(i)}(x)\prod_{j \in B_i}(x-a_j)}{p^{w_i}}.
\end{equation*}
Then $g_i$ is irreducible in $\IntZ$ by Remark \ref{remark:irreducible}.

Furthermore, each factorization of $G$ is of the form
 \begin{equation}\label{factorization:pattern}
G = \prod_{i=1}^{m_B}g_i \cdot \prod_{j\not \in \text{ Im }\theta}c_j(x) \cdot
\prod_{k \not \in \text{ Im }\sigma}d_k(x).
\end{equation}
Note that the length of the factorization in \eqref{factorization:pattern} is $s+t-m_B$.
\end{example}

\section{Generalizations}\label{sect:gen-results}
In this section we give lemmas generalizing the examples in Sections \ref{sect:typeI}
and \ref{sect:typeII}.
We begin with a generalization of Example \ref{exam:typeI}.

\begin{definition}
\begin{enumerate}
\item Let $I \neq \emptyset$ be a finite set and for $i\in I$,
let $g_i\in \Z[x]$ be primitive and irreducible in $\Z[x]$. Let
\begin{equation*}
f(x) = \frac{\prod_{i \in I}g_i(x)}{b} \in \IntZ
\end{equation*}
be irreducible in $\IntZ$ where $b > 1$ is a natural number.
 We call non-empty subsets $J_1, J_2 \subsetneqq I$ \emph{interchangeable} if
 $J_1 \cap J_2 = \emptyset$ and
 \begin{equation*}
 \fixdiv\Bigg(\prod_{i \in J_1}g_i(x)\cdot \prod_{i \in I\setminus J_2}g_i(x)\Bigg)
 = \fixdiv\Bigg(\prod_{i \in J_2}g_i(x)\cdot \prod_{i \in I\setminus J_1}g_i(x)\Bigg) = (b).
\end{equation*}
\item We call two non-empty disjoint index sets $J_1, J_2 \subsetneqq I$ \emph{element-disjoint}
if
\begin{equation*}
\{g_i \mid i \in J_1\} \cap \{g_j \mid j \in J_2\} = \emptyset.
\end{equation*}
\end{enumerate}
\end{definition}

\begin{example}
Consider the irreducible polynomial
\begin{equation*}
f = \frac{(x-1)(x-3)(x^2+4)}{4} \in \IntZ.
\end{equation*}
A quick check shows that
\begin{equation*}
\fixdiv((x-1)^2(x^2+4)) = \fixdiv((x-3)^2(x^2+4)) = (4).
\end{equation*}
Thus setting $g_1 = x-1, g_2 = x-3$ and $g_3 = x^2+4$, we see that the subsets
$J_1 = \{1\}$ and $J_2 = \{2\}$ of $I = \{1, 2, 3\}$ are interchangeable.
Furthermore, $J_1$ and $J_2$ are element-disjoint since they contain different elements.
\end{example}

\begin{lemma}\label{lemma:type1_non-abs}
Let $f(x) = \frac{\prod_{i \in I}g_i(x)}{b} \in \IntZ$
be irreducible in $\IntZ$, where $b > 1$ is a natural number, $I \neq \emptyset$
is a finite set and for $i\in I$, $g_i\in \Z[x]$ is primitive and
irreducible in $\Z[x]$.

If there exist two element-disjoint interchangeable
subsets $J_1, J_2 \subsetneqq I$, then $f$ is not absolutely irreducible.
\end{lemma}

\begin{proof}
Suppose $J_1, J_2 \subsetneqq I$ are element-disjoint and interchangeable.
Then for $k \geq 2,$
\begin{equation*}
f^k = \frac{\prod_{i \in J_1}g_i(x)\prod_{i \in I\setminus J_2}g_i(x)}{b}
\frac{\prod_{i \in J_2}g_i(x)\prod_{i \in I\setminus J_1}g_i(x)}{b}
\underbrace{\frac{\prod_{i \in I}g_i(x)}{b} \cdots
\frac{\prod_{i \in I}g_i(x)}{b}}_{k-2 \text{ copies}}
\end{equation*}
implies the existence of a factorization of $f^k$ essentially different from
$\underbrace{f \cdots f}_{k \text{ copies}}.$
\end{proof}

The next lemma tells us that we cannot have interchangeable subsets in
the case when the fixed divisor $b$ of the numerator of $f$ is a prime $p$.
We begin with a supporting definition.

\begin{definition}
Let $I \neq \emptyset$ be a finite set and for $i \in I$, let $f_i \in \Z[x]$ be primitive
and irreducible in $\Z[x]$. Let $p$ be a prime dividing $\fixdiv(\prod_{i\in I}f_i)$.
We say $f_{k}$ is \emph{indispensable for $p$} (among the polynomials $f_i$ with $i\in I$)
if there exists an integer $z$ such that $\val_p(f_k(z)) > 0$ and $\val_p(f_i(z)) = 0$
for all $i\neq k$. We call such a $z$ a \emph{witness} for $f_k$ being indispensable for $p$.
\end{definition}

\begin{example}
Consider the polynomials $f_1 = x, f_2 = x-1$ and $f_3 = x-2$ in $\Z[x]$. It is easy to check that
$\fixdiv(x(x-1)(x-2)) = 6$. Now note that $f_2 = x-1$ is indispensable for 2 since for all odd numbers $a$,
\begin{equation*}
\val_2(f_2(a)) = \val_2(a-1) > 0
\text{ and } \val_2(f_1(a)) = \val_2(f_3(a)) =0.
\end{equation*}
\end{example}
In this case any odd number is a witness for $x-1$ being indispensable for 2.
On the other hand, $x$ and $x-2$ are not indispensable for 2 since for any even number $b$,
\begin{equation*}
\val_2(f_1(b)) = \val_2(b) > 0
\text{ and } \val_2(f_3(b)) = \val_2(b-2) > 0.
\end{equation*}

\begin{lemma}\label{lemma:inter:p}
Let $I \neq \emptyset$ be a finite set and for $i\in I$, let $g_i\in \Z[x]$
be primitive and irreducible in $\Z[x]$. Let
\begin{equation*}
f(x) = \frac{\prod_{i \in I}g_i(x)}{p} \in \IntZ
\end{equation*}
be irreducible in $\IntZ$. Then there do not exist interchangeable
subsets of $I$.
\end{lemma}

\begin{proof}
Suppose $J_1, J_2 \subsetneqq I$ are disjoint. Now since $f$ is irreducible, every $g_i$ for $i \in I$
is indispensable for $p$ and this implies that $g_i \neq g_j$ for $i \neq j$. Thus $J_1$ and $J_2$ being disjoint,
are element-disjoint. Furthermore, if $r_i$ is a witness for $g_i$ being indispensable for $p$, then
\begin{equation*}
\val_p\left(\prod_{j \in I \setminus \{i\}}g_j(r_i)\right) = 0.
\end{equation*}
Now suppose $g_k$ for $k \in J_1$ is indispensable for $p$ with witness $r_k$. Then,
since $J_1$ and $J_2$ are element-disjoint,
it follows that
\begin{equation*}
\val_p\left(\prod_{j \in (I \setminus J_1)}g_j(r_k)\cdot \prod_{j \in J_2}g_j(r_k)\right) = 0.
\end{equation*}
Thus $\fixdiv\left(\prod_{j \in (I \setminus J_1)}g_j(x)\cdot \prod_{j \in J_2}g_j(x)\right)$
is not divisible by $p$. This shows that $J_1, J_2$ are not interchangeable because if they were, we
would have
\begin{equation*}
\fixdiv\left(\prod_{j \in (I \setminus J_1)}g_j(x)\cdot \prod_{j \in J_2}g_j(x)\right) = (p).
\end{equation*}
\end{proof}

The next lemma generalizes Example \ref{exam:typeII}. In Example \ref{exam:typeII},
setting $g_0 = c(x)$, $g_1 = d(x)$ and $g_i = x-a_{i-1}$ for $i = 2, \ldots, m+1$,
we can choose $\{0, 1\}$ for the index set $J$ in
Lemma \ref{lemma:type2_non-abs}.

\begin{lemma}\label{lemma:type2_non-abs}
Let $I \neq \emptyset$ be a finite set and for $i \in I$, let $g_i \in \Z[x]$ be
primitive and irreducible in $\Z[x]$.
Suppose
\begin{equation*}
 f = \frac{\prod_{i \in I}g_i(x)}{b} \in \IntZ
\end{equation*}
is irreducible in $\IntZ$, where $b > 1$. Let $\p$ be the set of prime divisors of $b$
and $b= \prod_{p \in \p} p^{e_p}$ be the prime factorization of $b$ with $e_p \in \mathbb{N}$.
If there exists a subset $\emptyset \neq J \subsetneqq I$ such that for all $p \in \p$,
for every integer $s$ that is a root mod $p$ of $\prod_{j\in J}g_j$, we have
\begin{equation}\label{ine}
\val_p\left(\prod_{j\in J}g_j(s)\right)> e_p,
\end{equation}
then $f$ is not absolutely irreducible.
\end{lemma}

\begin{proof}
Suppose there exists $\emptyset \neq J \subsetneqq I$ such that for all $p \in \p$,
inequality \ref{ine} is satisfied.
Let $n = \max\{e_p \mid p \in \p\}$. We claim that $f^{n+1}$ has a factorization
essentially different from $f \cdots f$. The existence of such a factorization follows from \ref{fact} below.

\begin{equation}\label{fact}
f^{n+1} = \frac{\left(\prod_{j \in J}g_j(x)\right)^{n}\left(\prod_{i \in I\setminus J}g_i(x)\right)^{n+1}}{b^{n+1}} \cdot
\prod_{j \in J}g_j(x).
\end{equation}
To see that the factor on the left is integer-valued,
let $p \in \p$ and $s \in \Z$. If $s$ is a root mod $p$ of
$\prod_{j\in J}g_j$, then

$$\val_p\left(\left(\prod_{j\in J}g_j(s)\right)^n \right) \ge n (e_p +1) = ne_p +n \ge (n+1)e_p = \val_p(b^{n+1})$$
On the other hand, if $s$ is not a root mod $p$ of $\prod_{j\in J}g_j$,
then
$$\val_p\left(\left(\prod_{j\in J}g_j(s)\right)^n \left(\prod_{i\in I\setminus J}g_i(s)\right)^{n+1}\right)
= \val_p\left(\left(\prod_{i\in I} g_i(s)\right)^{n+1} \right) \ge \val_p(b^{n+1}).$$

Finally, that the factorization \ref{fact} can be refined to a factorization
into irreducibles (necessarily essentially different from $f \cdots f$) follows
from the fact that $\Int(\Z)$ is atomic.
\end{proof}

We generalize Examples \ref{exam:typeIIc} and \ref{exam:typeIIb} in the next lemma.
In Example \ref{exam:typeIIc}, setting $g_0 = x^4 + x^3 +8$ and $g_1 = x-3$,
the index set $J$ in Lemma \ref{lemma:type2i_non-abs} was $\{0\}$.

\begin{lemma}\label{lemma:type2i_non-abs}
Let $I \neq \emptyset$ be a finite set and for $i \in I$, let $g_i \in \Z[x]$ be
primitive and irreducible in $\Z[x]$. Suppose
\begin{equation*}
f(x) = \frac{\prod_{i \in I}g_i(x)}{p^n} \in \IntZ
\end{equation*}
is irreducible in $\IntZ$, where $p$ is a prime and $n > 1$.

If there exists $J \subsetneqq I$ such that the following holds:
\begin{enumerate}
\item \label{item2:lemma:type2i_non-abs}$\Z = S \uplus T$ where $S$ and $T$ are each
a union of residue classes mod $p$ and such that
\begin{equation*}
\val_p\left(gcd\left(\prod_{i \in J}g_i(a)\Biggm\rvert a \in S\right)\right) > n \text{ and } \forall~ t \in T,~
\val_p\left(gcd\left(\prod_{i \in J}g_i(b)\Biggm\rvert b \in t + p\Z\right)\right) = e,
\end{equation*}
with $1 \leq e < n$, and for all $t \in T$
\begin{equation}\label{ine:type2i}
\val_p\left(gcd\left(\prod_{i \in J}g_i(b)\Biggm\rvert b \in t + p\Z\right)\right) +
\val_p\left(gcd\left(\prod_{i \in I\setminus J}g_i(b)\Biggm\rvert b \in t + p\Z\right)\right) \geq n.
\end{equation}
\end{enumerate}
Then $f$ is not absolutely irreducible.
\end{lemma}

\begin{proof}
Suppose there exists a subset $J \subsetneqq I$ such that \ref{item2:lemma:type2i_non-abs} holds.

It follows from inequality \eqref{ine:type2i} and $f$ being irreducible that
\begin{equation*}
\min_{t \in T}\left\{\val_p\left(gcd\left(\prod_{i \in I\setminus J}g_i(b)\Biggm\rvert b \in t + p\Z\right)\right)\right\} = n-e.
\end{equation*}
Now let
\begin{equation*}
m = \val_p\left(gcd\left(\prod_{i \in J}g_i(a)\Biggm\rvert a \in S\right)\right) > n.
\end{equation*}
We set $k=m-e$ and claim that $f^k$ has a factorization essentially different from
$f \cdots f$. This factorization follows from
 \begin{equation*}
f^k = \frac{(\prod_{i \in J}g_i)^{n-e}(\prod_{i \in I \setminus J} g_i)^{k}}{p^{(n-e)m}}
\cdot \left(\frac{\prod_{i \in J} g_i}{p^{e}}\right)^{m-n}.
\end{equation*}

In fact, for each $t \in T$, let
\begin{equation*}
l = \val_p\left(gcd\left(\prod_{i \in J}g_i(b)\Biggm\rvert b \in t + p\Z \right)^{n-e}\right) +
\val_p\left(gcd\left(\prod_{i \in I \setminus J}g_i(b)\Biggm\rvert b \in t + p\Z \right)^k\right).
\end{equation*}
Then $l = e(n-e) + k(n-e) = e(n-e) + (m-e)(n-e) = (n-e)m$.

Furthermore
\begin{equation*}
\val_p\left(gcd\left(\prod_{i \in J}g_i(a)\Biggm\rvert
a \in S\right)^{n-e}\right) = (n-e)m.
\end{equation*}
Moreover, $k-(n-e)= m-n$ and $(n-e)m + (m-n)e = (m-e)n = kn$.
\end{proof}

\section*{Acknowledgement}
I would like to thank my advisor Prof. Sophie Frisch for her comments and guidance.
The author is supported by the Austrian Science Fund~(FWF):~P~30934.

\bibliographystyle{plain}
\bibliography{references}

\end{document}